\documentclass{amsart}
\usepackage{amssymb,amsmath,bbm}
\usepackage{pdflscape}
\usepackage{mathrsfs}
\usepackage{tikz}
\usepackage{graphicx}
\usepackage{caption}
\usepackage{subcaption}
\usepackage{rotating}
\usepackage{afterpage}
\usepackage{hyperref}
\usepackage{url}
\usepackage{epsfig}

\usetikzlibrary{calc,intersections, arrows, shapes}

\newcommand{\RR}{\mathbb{R}}

\newcommand{\ZZ}{\mathbb{Z}}
\newcommand{\NN}{\mathbb{N}}

\newcommand{\VV}{\mathcal{V}}

\newcommand{\PSD}{\mathcal{S}_+}
\newcommand{\rank}{\textup{rank}\,}

\newcommand{\psdrank}{\textup{rank}_{\textup{psd}}}

\newcommand{\conv}{\textup{conv}}

\newcommand{\sqrtrank}{\textup{rank}_{\! \! {\sqrt{\ }}}\,}

\def\ext{\ensuremath{\textup{ext}}}

\newtheorem{theorem}{Theorem}[section]

\newtheorem{lemma}[theorem]{Lemma}
\newtheorem{corollary}[theorem]{Corollary}
\newtheorem{proposition}[theorem]{Proposition}
\theoremstyle{definition}
\newtheorem{definition}[theorem]{Definition}
\newtheorem{example}[theorem]{Example}

\theoremstyle{remark}

\newtheorem{question}[theorem]{Question}

\title{Spectrahedral Lifts of Convex Sets}

\author[Thomas]{Rekha R. Thomas}
\address{Department of Mathematics, University of Washington, Box
  354350, Seattle, WA 98195, USA} \email{rrthomas@uw.edu}

\thanks{The author was partially supported by the U.S. National Science Foundation grant DMS-1719538. 
This paper was written while the author was in residence at the Mathematical Sciences Research Institute in Berkeley, California, during the Fall 2017 semester, and based on work supported by the National Science Foundation under Grant No. 1440140.}
\date{\today}

\begin{document}

\begin{abstract}
Efficient representations of convex sets are of crucial importance for many algorithms that work with them. It is well-known that 
sometimes, a complicated convex set can be expressed as the projection of a much simpler set in higher 
dimensions called a {\em lift} of the original set. This is a brief survey of recent developments in the 
topic of lifts of convex sets. Our focus will be on lifts that arise from affine slices of real positive semidefinite 
cones known as {\em psd} or {\em spectrahedral lifts}. The main result is that projection representations 
of a convex set are controlled by factorizations, through closed convex cones, of an operator that comes from the convex set.
This leads to several research directions and results that lie at the intersection of convex geometry, combinatorics, real algebraic geometry, optimization, computer science and more.
\end{abstract}
\maketitle

%%%%%%%%%%%%%%%%%%%%%%%%%%%%%%%%%%%%%%%
\section{Introduction} \label{sec:introduction}

Efficient representations of convex sets are of fundamental importance in many areas of mathematics. An old idea from  
optimization for creating a compact representation of a convex set is to express it 
as the projection of a higher-dimensional set that might potentially be simpler, see for example 
\cite{CCZSurvey}, \cite{BenTalNemirovski}. 
In many cases, this technique offers surprisingly  
compact representations of the original convex set. We present the basic questions that arise in the 
context of projection representations, provide some answers, pose more questions, and examine the current limitations and challenges.

As a motivating example, consider a full-dimensional convex polytope $P \subset \RR^n$. Recall that $P$ can be 
expressed either as the convex hull of a finite collection of points in $\RR^n$ or as the intersection of a finite set of linear halfspaces. The minimal set of points needed in the convex hull representation are the {\em vertices} of $P$, and the 
irredundant inequalities needed are in bijection with the {\em facets} (codimension-one faces) of $P$. Therefore, if the number of facets of $P$ is 
exponential in $n$, then the linear inequality representation of $P$ is of size exponential  in $n$.
The complexity of optimizing a linear function over $P$ depends on the size of its inequality representation and hence it is worthwhile to ask if efficient 
inequality representations can be obtained through some indirect means such as projections.
We illustrate the idea on two examples.

\begin{example} \label{ex:crosspolytope}
The $n$-dimensional {\em crosspolytope} $C_n$ is the convex hull of the standard unit vectors $e_i \in \RR^n$ 
and their negatives  \cite[Example 0.4]{ZieglerBook}. For example, $C_2$ is a square and $C_3$ is an octahedron, see Figure~\ref{fig:permutahedron and octahedron}. Written in terms of inequalities, 
$$C_n = \{ x \in \RR^n \,:\, \pm x_1 \pm x_2 \pm \cdots \pm x_n \leq 1 \}$$
and all $2^n$ inequalities listed are needed as they define facets of $C_n$. 
However, $C_n$ is also the projection onto $x$-coordinates 
of the polytope 
$$Q_n =\left\{ (x,y) \in \RR^{2n} \,:\, \sum_{i=1}^n y_i = 1, \,\, -y_i \leq x_i \leq y_i \,\,\, \forall i=1,\ldots,n \right\}$$
which involves only $2n$ inequalities and one equation. 
\qed
\end{example}

\begin{figure} 
\includegraphics[scale=0.2]{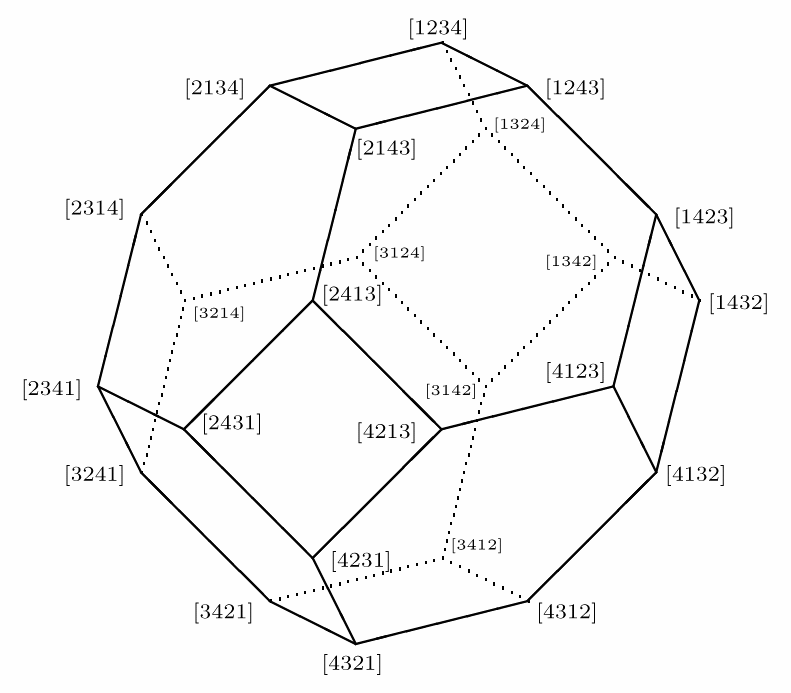} 
\hspace{1cm}
\includegraphics[scale=0.25]{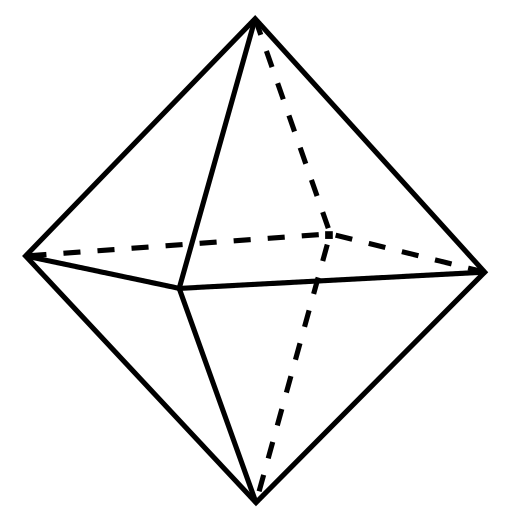}
\caption{The permutahedron $\Pi_4$ and the crosspolytope $C_3$. \label{fig:permutahedron and octahedron}}
\end{figure}

\begin{example} \label{ex:permutahedron}
The {\em permutahedron} $\Pi_n$ is the $(n-1)$-dimensional polytope that is the convex hull of all vectors obtained by permuting the coordinates of the $n$-dimensional vector $(1,2,3, \ldots, n)$. It has $2^n - 2$ facets, each indexed by a proper subset of $[n]:=\{1,2,\ldots,n\}$ \cite[Example 0.10]{ZieglerBook}. In \cite{Goemans}, Goemans used sorting networks to show that $\Pi_n$ is the linear image of a polytope $Q_n$ that has $\Theta(n \, \log \, n)$ variables and facets, and also argued that one cannot do better. \qed
\end{example}

The key takeaway from the above examples is that one can sometimes find efficient linear representations of polytopes if extra variables are allowed; a complicated polytope $P \subset \RR^n$ might be the linear projection of a polytope $Q \subset \RR^{n+k}$ with many fewer facets. To be considered efficient, both $k$ and the number of facets of $Q$ must be polynomial functions of $n$. Such a polytope $Q$ is called a {\em lift} or {\em extended formulation} of $P$. Since optimizing a linear function over $P$ is equivalent to optimizing the same function over a lift of it, these projection representations offer the possibility of efficient algorithms for 
linear programming over $P$.

Polytopes are special cases of closed convex sets and one can study closed convex lifts in this more general context. 
All convex sets are slices of closed convex cones by affine planes and hence we will look at lifts of convex sets 
that have this form. Formally, given a closed convex cone $K \subset \RR^m$, an affine plane $L \subset \RR^m$, and a convex set $C \subset \RR^n$, we say that $K \cap L$ is a $K$-{\em lift} of $C$ if $C = \pi(K \cap L)$ for some linear map $\pi \,:\, \RR^m \rightarrow \RR^n$.  Recall that every polytope is an affine slice of a nonnegative orthant $\RR^k_+$ and hence polyhedral lifts of polytopes, 
as we saw in Examples~\ref{ex:crosspolytope} and \ref{ex:permutahedron},  are special cases of cone lifts. 
A polytope can also have non-polyhedral lifts.

The main source of non-polyhedral lifts in this paper will come from the {\em positive semidefinite cone} $\PSD^k$ of $k \times k$ real symmetric positive semidefinite (psd) matrices. If a matrix $X$ is psd, we write $X \succeq 0$.
An affine slice of $\PSD^k$ is called a {\em spectrahedron} of size $k$. If a spectrahedron (of size $k$) is a lift of a convex set $C$, we say that $C$ admits a {\em spectrahedral or psd lift} (of size $k$). It is also common to say that $C$ is {\em sdp representable} or a {\em projected spectrahedron} or a {\em spectrahedral shadow}. 
Note that a spectrahedron in $\PSD^k$ can also be written in the form 
$$\left\{ x \in \RR^t \,:\, A_0 + \sum_{i=1}^t  A_i x_i \succeq 0 \right\}$$ 
where $A_0, A_1, \ldots, A_t$ are real symmetric matrices of size $k$. 

\begin{example} \label{ex:square}
The square $P \subset \RR^2$ with vertices $(\pm 1, \pm 1)$ can be expressed as the projection of a spectrahedron as follows:
$$ P = \left\{ (x,y) \in \RR^2 \,:\, \exists \, z \in \RR \textup{ s.t. } \begin{pmatrix} 1 & x & y \\ x & 1 & z\\ y & z & 1   \end{pmatrix} \succeq 0 \right \}.$$
\begin{figure} 
\includegraphics[scale=0.35]{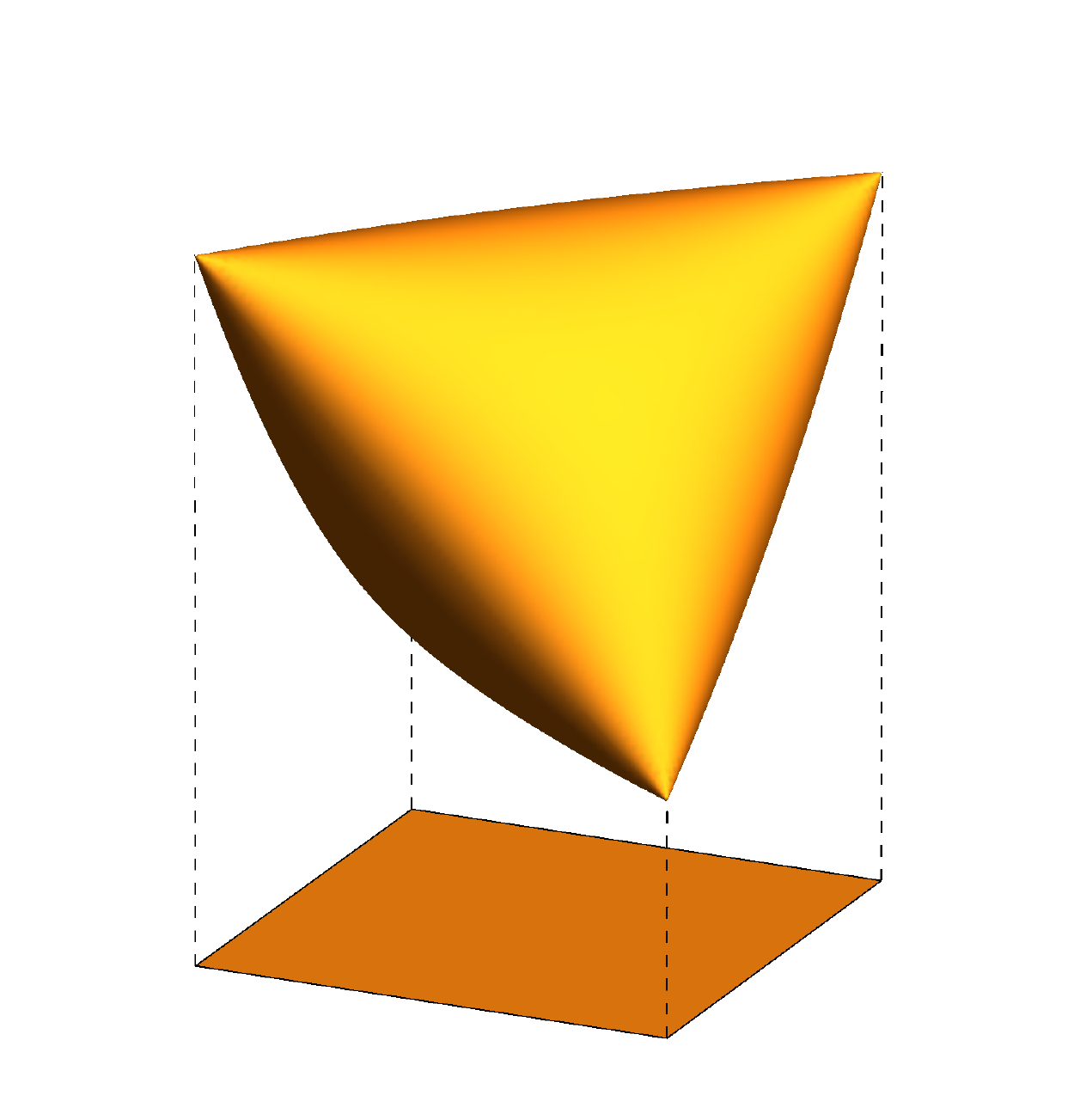}
\caption{A spectrahedral lift of a square. \label{fig:spectrahedral lift of square}
}
\end{figure}
The spectrahedral lift in this example is known as the {\em elliptope} and is shown in Figure~\ref{fig:spectrahedral lift of square}. It consists of all $X \in \PSD^3$ such that $X_{ii} = 1$ for $i=1,2,3$. \qed
\end{example}

\begin{example} \label{ex:stable set polytope}
Given a graph $G=([n],E)$ with vertex set $[n]$ and edge set $E$, a collection $S \subseteq [n]$ is called a {\em stable set} if for each $i,j \in S$, the pair $\{i,j\} \not \in E$. Each stable set $S$ is uniquely identified by its 
incidence vector $\chi^S \in \{0,1\}^n$ defined as $(\chi^S)_i = 1$ if $i \in S$ and $0$ otherwise. The {\em stable set polytope} of $G$ is 
$$ \textup{STAB}(G) := \conv\{ \chi^S \,:\, S \textup{ stable set in } G \}$$
where $\conv$ denotes convex hull.
For $x=\chi^S$, consider the rank one matrix in $\PSD^{n+1}$

$$\begin{pmatrix} 1 \\ x \end{pmatrix} \begin{pmatrix} 1 & x^\top \end{pmatrix} = 
\begin{pmatrix} 1 & x^\top \\ x & xx^\top \end{pmatrix} = \begin{pmatrix} 1 & x^\top \\ x & U \end{pmatrix}.$$

Since $\chi^S \in \{0,1\}^n$, $U_{ii} = x_i$ for all $i \in [n]$, and since $S$ is stable, $U_{ij} = 0$ for all $\{i,j\} \in E$. Therefore, the convex set 

$$ \textup{TH}(G) := \left\{ x \in \RR^n \,:\, 
\begin{array}{l} 
\exists U \in \PSD^n \textup{ s.t. } \begin{pmatrix} 1 & x^\top \\
x & U \end{pmatrix} \succeq 0, \\
U_{ii} = x_i \,\,\forall i \in [n],\\
U_{ij} = 0 \,\,\forall \{i,j\} \in E 
\end{array}
\right \}$$
known as the {\em theta body}  of $G$, contains all the vertices of $\textup{STAB}(G)$, and hence by convexity, all of 
$\textup{STAB}(G)$. In general, this containment is strict.  The theta body $\textup{TH}(G)$ is the projection onto $x$-coordinates of the set of all matrices in $\PSD^{n+1}$ whose entries satisfy a set of linear constraints. The latter is a spectrahedron.

Theta bodies of graphs were defined by Lov{\'a}sz in \cite{LovaszTheta}. He proved that 
$\textup{STAB}(G) = \textup{TH}(G)$ if and only if $G$ is a perfect graph. Even for perfect graphs, 
$\textup{STAB}(G)$ can have exponentially many facets, but by Lov{\'a}sz's result, it admits a spectrahedral lift of 
size $n+1$. \qed
\end{example}

We close the introduction with a psd lift of a non-polytopal convex set.
Since polyhedra can only project to polyhedra, any lift of a non-polyhedral convex set is necessarily non-polyhedral.

\begin{example} \label{ex:ball}
Let $X$ be the $n \times n$ symbolic matrix with entries $x_1, \ldots, x_{n^2}$ written consecutively along its $n$ rows.  and let $I_n$ denote the $n \times n$ identity matrix. Consider the spectrahedron of size $2n$ defined by the conditions 
$$ \begin{pmatrix} 
                                                 Y  & X \\ 
                                                         & \\
                                                         X^\top & I_n 
                                                \end{pmatrix}  \succeq 0, \,\,\,\textup{trace}(Y) = 1.$$
The psd condition is equivalent to $Y - X X^\top \succeq 0$ via Schur complement.  Taking the trace on both sides we get $1 = \textup{trace}(Y) \geq \textup{trace}(X X^\top) = \sum_{i=1}^{n^2} x_i^2$. 
Thus, the projection of the above spectrahedron onto $x = (x_1, \ldots, x_{n^2})$ is contained in the unit ball 
$$B_{n^2} := \{ (x_1,\ldots,x_{n^2}) \,:\, \sum x_i^2 \leq 1 \}.$$
On the other hand, for any $x$ on the boundary of $B_{n^2}$, the matrix 
$$   \begin{pmatrix} 
                                                 XX^\top  & X \\ 
                                                         & \\
                                                         X^\top & I_n 
                                                \end{pmatrix}$$ 
                                                lies in the above spectrahedron and projects onto $x$. We conclude that 
                                                $B_{n^2}$ has a spectrahedral lift of size $\textup{O}(n)$.  
\qed
\end{example}

In many of the above cases, projections offer a more compact representation of the convex set in question 
compared to the natural representation the set came with. Two fundamental questions we can ask now are the following.

\begin{question} \label{qn:existence of cone lift}
Given a convex set $C \subset \RR^n$ and a closed convex cone $K \subset \RR^m$, does $C$ admit a $K$-lift? 
\end{question}

\begin{question}
If $K$ comes from a family of cones $\{K_t \subset \RR^t\}$ such as the set of all positive orthants or the set of all psd cones, what is the smallest $t$ for which $C$ admits a $K_t$-lift?
The smallest such $t$ is a measure of complexity of $C$.
\end{question}

We will address both these questions and discuss several further related directions and results. In Section~\ref{sec:factorization for convex sets}   we prove that the existence of a $K$-lift for a convex set $C$ is 
controlled by the existence of a $K$-factorization of an operator associated to $C$. This result specializes nicely to polytopes as we will see in Section~\ref{sec:factorization for polytopes}. These factorization theorems generalize a celebrated result of Yannakakis 
\cite{Yannakakis} about polyhedral 
lifts of polytopes. The rest of the sections are focussed on spectrahedral lifts of convex sets. In Section~\ref{sec:psd rank} we define the notion 
of positive semidefinite rank  (psd rank) of a convex set and explain the known bounds on this invariant.  We also mention recent results about psd ranks of certain families of convex sets. The psd rank of an $n$-dimensional polytope is known to be at least $n+1$. In Section~\ref{sec:psd minimality}, we explore the class of polytopes that have this minimum possible psd rank. 
We conclude in Section~\ref{sec:sos} with the basic connections between sum of squares polynomials and spectrahedral lifts. 
We also describe the recent breakthrough by Scheiderer that provides the first examples of convex semialgebraic sets 
that do not admit spectrahedral lifts.

%%%%%%%%%%%%%%%%%%%%%%%%%%%%%%%%%%%%%%%%%%%%%%%%%%%
\section{The Factorization Theorem for Convex Sets}
\label{sec:factorization for convex sets}

A convex set is called a {\em convex body} if it is compact and contains the 
origin in its interior. For simplicity, we will always assume that all our convex sets 
are convex bodies. 
Recall that the {\em polar} of a convex set $C \subset \RR^n$ is the set
\[
C^\circ=\{y \in \RR^n: \langle x,y \rangle \leq 1, \ \ \forall x \in C \}.
\] 
Let $\ext(C)$ denote the set of {\em extreme points} of
$C$, namely, all points $p \in C$ such that if $p = (p_1 + p_2)/2$, with $p_1,p_2 \in C$, then 
$p=p_1=p_2$.  Both $C$ and $C^\circ$ are convex hulls of their  
respective extreme points. Consider the operator $S:\RR^n\times \RR^n
\rightarrow \RR$ defined by $S(x,y)=1-\left<x,y\right>$. The {\em slack operator} $S_C$, of a convex set $C \subset \RR^n$, 
is the restriction of the operator $S$ to $\ext(C) \times \ext(C^\circ)$. Note that the range of $S_C$ is contained in 
$\RR_+$, the set of nonnegative real numbers.

\begin{definition} \label{def:K-lift}
Let $K \subset \RR^m$ be a full-dimensional closed convex cone and $C \subset \RR^n$ a full-dimensional 
convex body. A {\em $K$-lift} of $C$ is a set $Q=K \cap L$,
where $L \subset \RR^m$ is an affine subspace, and $\pi:\RR^m \rightarrow \RR^n$ is a linear map 
such that $C=\pi(Q)$. If
$L$ intersects the interior of $K$ we say that $Q$ is a {\em
  proper $K$-lift} of $C$.
\end{definition}

We will see that the existence of a $K$-lift of $C$ is intimately related to properties of the slack operator $S_C$.
Recall that the {\em dual} of a 
closed convex cone $K \subset \RR^m$ is 
$$K^*=\{y \in \RR^m: \langle x,y \rangle \geq 0, \ \ \forall x \in K
\}.$$ A cone $K$ is \emph{self-dual} if $K^* = K$. The cones $\RR_+^n$ and $\PSD^k$ are self-dual.

\begin{definition}
Let $C$ and $K$ be as in Definition~\ref{def:K-lift}. We say that the slack operator 
$S_C$ is
\emph{$K$-factorizable} if there exist maps (not necessarily linear)
$$A:\ext(C)\rightarrow K \,\,\,\,\textup{and} \,\,\,\, B:\ext(C^\circ)\rightarrow K^{*}$$
such
that $S_C(x,y)=\left<A(x),B(y)\right>$ for all $(x,y) \in \ext(C)
\times \ext(C^\circ)$.
\end{definition}

We can now characterize the existence of a $K$-lift of $C$ in terms of the operator $S_C$, answering 
Problem~\ref{qn:existence of cone lift}. The proof 
relies on the theory of {\em convex cone programming} which is the problem of optimizing a linear 
function over an affine slice of a closed convex cone, see \cite{BenTalNemirovski}, or \cite[\S 2.1.4]{SIAMBook} for a quick introduction.

\begin{theorem}  \cite[Theorem 1]{GPT-lifts}
\label{thm:general_Yannakakis}
If $C$ has a proper $K$-lift then $S_C$ is
$K$-factorizable. Conversely, if $S_C$ is $K$-factorizable then $C$
has a $K$-lift.
\end{theorem}

\begin{proof}
Suppose $C$ has a proper $K$-lift. Then there exists an affine space $L = w_0+L_0$ in $\RR^m$ ($L_0$ is a linear subspace) and a linear map $\pi:\RR^m
\rightarrow \RR^n$ such that $C = \pi(K \cap L)$ and $w_0 \in
\textup{int}(K)$. Equivalently, 
\[
C = \{ x \in \RR^n : x = \pi(w), \quad w \in K \cap (w_0 + L_0) \}.
\]
We need to construct the maps $A \,:\, \textup{ext}(C) \rightarrow K$ and 
$B:\textup{ext}(C^\circ)\rightarrow K^*$ that factorize the
slack operator $S_C$, from the $K$-lift of $C$. 
For $x_i \in \textup{ext}(C)$, define $A(x_i) :=
w_i$, where $w_i$ is any point in the non-empty convex set 
$\pi^{-1} (x_i) \cap K \cap L$.

Let $c$ be an extreme point of $C^\circ$. Then $\textup{max} \{ \,
\langle c, x \rangle \,:\, x \in C \, \} = 1$ since $\langle c, x
\rangle \leq 1$ for all $x \in C$, and if the maximum was smaller than
one, then $c$ would not be an extreme point of $C^{\circ}$.  Let $M$
be a full row rank matrix such that $\textup{ker} \,M = L_0$. Then the
following hold:
$$
\begin{array}{cccccc}
\begin{array}{c} 1 = \\ \\ \\ \end{array} &
\begin{array}{c} \textup{max} \langle c,x \rangle \\  x \in C \\ \\ \end{array} & 
\begin{array}{c} = \\ \\ \\ \end{array} &
\begin{array}{c} \textup{max} \langle c, \pi(w) \rangle \\ w  \in K \cap (w_0 + L_0) \\ \\ \end{array} & 
\begin{array}{c} = \\ \\ \\ \end{array} &
\begin{array}{c} \textup{max} \langle \pi^*(c), w \rangle \\ M w  = M w_0 \\ w \in K \end{array} 
\end{array}
$$
Since $w_0$ lies in the interior of $K$,  by Slater's condition we have strong duality for the above cone program, and we get
$$ 1 = \textup{min} \,\langle M w_0, y \rangle \,:\, M^T y - \pi^*(c) \in K^*$$
with the minimum being attained.
Further, setting $z = M^T y$ we have that 
$$ 1 = \textup{min} \, \langle w_0, z  \rangle \,:\,   z - \pi^*(c) \in K^*, \, z \in L_0^\perp $$
with the minimum being attained. Now define 
$B \,:\, \textup{ext}(C^\circ) \rightarrow K^*$ as the map that sends $y_i \in \textup{ext}(C^\circ)$ to 
$B(y_i) := z - \pi^*(y_i)$, where $z$ is any point in the nonempty convex set $L_0^\perp
\cap (K^* + \pi^*(y_i))$ that satisfies $\langle w_0,z \rangle =1$. Note that for such a $z$, $\langle w_i, z \rangle = 1$ for all $w_i \in L$.
Then $B(y_i) \in K^*$, and for an $x_i \in \textup{ext}(C)$,
\begin{align*}
\langle x_i, y_i \rangle &=
\langle \pi (w_i), y_i \rangle  = 
\langle w_i, \pi^* (y_i )\rangle =
\langle w_i , z - B(y_i) \rangle  \\ 
&= 1 - \langle w_i, B(y_i) \rangle =
1 - \langle A(x_i), B(y_i) \rangle.
\end{align*}

Therefore, $S_C(x_i,y_i) = 1 - \langle x_i, y_i \rangle = \langle A(x_i), B(y_i) \rangle$ for all $x_i \in \textup{ext}(C)$ and 
$y_i \in \textup{ext}(C^\circ)$. 

Suppose now $S_C$ is $K$-factorizable, i.e., there exist maps
$A:\ext(C)\rightarrow K$ and $B:\ext(C^\circ)\rightarrow K^*$ such
that $S_C(x,y)=\left<A(x),B(y)\right>$ for all $(x,y) \in \ext(C)
\times \ext(C^\circ)$. Consider the affine space
\[
L=\{(x,z) \in \RR^n \times \RR^m : 1 - \langle x, y \rangle =\left<z,B(y)\right>,
\ \forall \,\,y \in \ext(C^\circ)\},
\] 
and let $L_K$ be its coordinate projection into $\RR^m$. Note that $0 \not \in L_K$ since otherwise, 
there exists $x \in \RR^n$ such that $1- \langle x, y \rangle = 0$ for all $y \in \ext(C^\circ)$ 
which implies that $C^\circ$ lies in the affine hyperplane $\langle x, y \rangle = 1$. This is a contradiction 
since $C^\circ$ contains the origin. Also, $K \cap L_K \neq \emptyset$ since for each $x \in \ext(C)$, 
$A(x) \in K \cap L_K$ by assumption.

Let $x$ be some point in $\RR^n$ such that there exists some $z \in K$
for which $(x,z)$ is in $L$. Then, for all extreme points $y$ of
$C^\circ$ we will have that $1-\langle x,y \rangle$ is nonnegative. This implies,
using convexity, that $1-\left< x ,y \right>$ is nonnegative for all
$y$ in $C^\circ$, hence $x \in (C^\circ)^\circ=C$.

We now argue that this implies that for each $z \in K \cap L_K$ there exists a unique $x_z
\in \RR^n$ such that $(x_z,z) \in L$. That there is one, comes
immediately from the definition of $L_K$. Suppose now that there is
another such point $x_{z}'$. Then $(t x_{z}+(1-t)x_{z}',z) \in L$ for
all reals $t$ which would imply that the line through $x_{z}$ and
$x_{z}'$ would be contained in $C$, contradicting our assumption that $C$ is 
compact.

The map that sends $z$ to $x_{z}$ is therefore well-defined in $K \cap
L_K$, and can be easily checked to be affine. Since the origin is not in $L_K$, we can extend it to
a linear map $\pi:\RR^m \rightarrow \RR^n$.  To finish the proof it
is enough to show $ C = \pi(K \cap L_K)$. We have already seen that
$\pi(K \cap L_K) \subseteq C$ so we just have to show the reverse
inclusion. For all extreme points $x$ of $C$, $A(x)$ belongs to $K \cap L_K$, and 
therefore, $x=\pi(A(x)) \in \pi(K \cap L_K)$. Since $C =
\conv(\ext(C))$ and $\pi(K \cap L_K)$ is convex, $C \subseteq \pi(K \cap L_K)$.
\end{proof}

The asymmetry in the two directions of Theorem~\ref{thm:general_Yannakakis} disappears for many nice cones 
including $\RR^k_+$ and $\PSD^k$. For more on this, see \cite[Corollary 1]{GPT-lifts}. In these nice cases, 
$C$ has a $K$-lift if and only if $S_C$ has a $K$-factorization. 
Theorem~\ref{thm:general_Yannakakis} generalizes the 
original factorization theorem of Yannakakis for polyhedral lifts of polytopes \cite[Theorem 3, \S4]{Yannakakis} to 
arbitrary cone lifts of convex sets. 

Recall that in the psd cone $\PSD^k$, the inner product $\langle A, B \rangle = \textup{trace}(AB)$.

\begin{example} \label{ex:disc}
The unit disk $C \subset \RR^2$ is a spectrahedron in $\PSD^2$ as follows 
\[
C=\left\{
(x,y) \in \RR^2 : \left( \begin{array}{cc}
1+x & y \\
y & 1-x
\end{array}\right) \succeq 0
\right\}, 
\]
and hence trivially has a $\PSD^2$-lift.
This means that the slack operator $S_C$ must have a $\PSD^2$ factorization. Since $C^\circ=C$, $\textup{ext}(C) = \textup{ext}(C^\circ) = \partial C$, and so we have to find maps $A,B: \textup{ext}(C) \rightarrow \PSD^2$ such
that for all $(x_1,y_1), (x_2,y_2) \in \ext(C)$,
\[
\left<A(x_1,y_1),B(x_2,y_2)\right> = 1 -x_1x_2 - y_1y_2.
\] 
This is accomplished by the maps
\[
A(x_1,y_1)=\left( \begin{array}{cc}
1+x_1 & y_1 \\
y_1 & 1-x_1
\end{array}\right)
\]
and
\[
B(x_2,y_2)=\frac{1}{2}\left( \begin{array}{cc}
1-x_2 & -y_2 \\
-y_2 & 1+x_2
\end{array}\right)
\]
which factorize $S_C$ and are positive semidefinite in their domains. 
\qed
\end{example}

\begin{example}
Consider the spectrahedral lift of the unit ball $B_{n^2}$ from Example~\ref{ex:ball}. Again, we have that
$\textup{ext}(B_{n^2}) = \textup{ext}(B_{n^2}^\circ) = \partial B_{n^2}$. 
The maps $$A(x) = 
\begin{pmatrix} 
XX^\top & X \\ X^\top & I_n
\end{pmatrix}, \,\,\,\,
B(y) = \frac{1}{2}
\begin{pmatrix} 
I_n & -Y \\ -Y^\top & YY^\top
\end{pmatrix}$$
where $X$ is defined as in Example~\ref{ex:ball} and $Y$ is defined the same way, 
offer a $\PSD^{2n}$-factorization of the slack operator of $B_{n^2}$.  \qed
\end{example}

The existence of cone lifts of convex bodies is preserved under many geometric operations \cite[Propositions 1 and 2]{GPT-lifts}. 
For instance, if $C$ has a $K$-lift, then so does any compact image of $C$ under a projective transformation. An  
elegant feature of this theory is that the existence of lifts is invariant under polarity/duality; $C$ has a $K$-lift if and only if $C^\circ$ has a $K^\ast$-lift. In particular, if $C$ has a spectrahedral lift of size $k$, then so does $C^\circ$.
%%%%%%%%%%%%%%%%%%%%%%%%%%%%%%%%%%%%%%%%%%%%%%%%%%%%%%
\section{The Factorization Theorem for Polytopes}
\label{sec:factorization for polytopes}

When the convex body $C$ is a polytope, 
Theorem~\ref{thm:general_Yannakakis} becomes rather simple. This specialization also 
appeared in \cite{FMPTW}. 

\begin{definition}
Let $P$ be a full-dimensional polytope in $\RR^n$ with vertex set 
$V_P=\{p_1,\ldots,p_v\}$ and an irredundant inequality representation
\[
P=\{ x \in \RR^n : h_1(x) \geq 0 , \ldots, h_f(x) \geq 0\}.
\]
Since $P$ is a convex body, we may assume that the constant in each $h_j(x)$ is $1$.
The {\em slack matrix of $P$} is the nonnegative $v \times f$ matrix whose $(i,j)$-entry
is $h_j(p_i)$, the {\em slack} of vertex $p_i$ in the facet inequality $h_j(x) \geq 0$.
\end{definition}

When $P$ is a polytope, 
$\ext(P)$ is just $V_P$, 
and $\ext(P^\circ)$ is in bijection with $F_P$, the set of facets of $P$. 
The facet $F_j$ is defined by $h_j(x) \geq 0$ and $f := \left|F_P\right|$.
Then the slack operator $S_P$ is the map from 
$V_P \times F_P$ to $\RR_+$ that sends the vertex facet pair $(p_i,F_j)$ to
$h_j(p_i)$.  Hence, we may identify the slack operator of $P$ with
the slack matrix of $P$ and use $S_P$ to also denote this
matrix.  Since the facet inequalities of $P$ are only unique up to multiplication by positive scalars, the matrix 
$S_P$ is also only unique up to multiplication of its columns by positive scalars. Regardless, we will call $S_P$, derived from the given presentation of $P$, {\em the} slack 
matrix of $P$. 

\begin{definition} 
\label{def:K-factorization of nonnegative matrices}
Let $M = (M_{ij}) \in \RR_+^{p \times q}$ be a nonnegative matrix and $K$ a
closed convex cone. Then a $K$-{\em factorization} of $M$ is a pair of
ordered sets $\{a^1, \ldots, a^p\} \subset K$ and $\{b^1, \ldots, b^q\} \subset K^*$
such that $\langle a^i, b^j \rangle = M_{ij}$.
\end{definition}

Note that $M \in \RR^{p \times q}_+$ has a $\RR_+^k$-factorization if and only
if there exist a $p \times k$ nonnegative matrix $A$ and a $k \times
q$ nonnegative matrix $B$ such that $M=AB$, called a {\em nonnegative factorization} of $M$.  Definition~\ref{def:K-factorization of nonnegative matrices} 
generalizes nonnegative factorizations of nonnegative matrices to cone factorizations.

\begin{theorem} \label{thm:general_Yannakakis_for_polytopes}
If a full-dimensional polytope $P$ has a proper $K$-lift then every slack matrix of $P$ admits a 
$K$-factorization. Conversely, if some slack matrix of $P$ has a
$K$-factorization then $P$ has a $K$-lift.
\end{theorem}

Theorem~\ref{thm:general_Yannakakis_for_polytopes} is a direct translation of 
Theorem~\ref{thm:general_Yannakakis} 
using the identification between the slack operator of $P$ and the slack
matrix of $P$. The original theorem of Yannakakis \cite[Theorem~3, \S4]{Yannakakis} proved this result 
in the case where $K = \RR_+^k$.

\begin{example} \label{ex:hexagon}
Consider the regular hexagon with inequality description 
$$ H=\left \{ (x_1,x_2) \in \RR^2 \,:\, \left(\begin{array}{cc}
1 & \sqrt{3}/3 \\
0 & 2\sqrt{3}/3 \\
-1 & \sqrt{3}/3 \\
-1 & -\sqrt{3}/3 \\
0 & -2\sqrt{3}/3 \\
1 & -\sqrt{3}/3 
\end{array}
\right)
\left(\begin{array}{c}
x_1 \\
x_2
\end{array}
\right) \leq 
\left(
\begin{array}{c}
1 \\
1\\
1 \\  
 1 \\
 1\\
 1
\end{array}
\right) \right \}. $$ 
We will denote the coefficient matrix by $F$ and the right hand side
vector by $d$.  It is easy to check that $H$ cannot be the projection of an 
affine slice of $\RR^k_+$ for $k < 5$. Therefore, we ask whether it can be the linear image 
of an affine slice of $\RR_+^5$. Using
Theorem~\ref{thm:general_Yannakakis_for_polytopes} this is equivalent to asking if 
the slack matrix of the hexagon, 
$$ S_H := \left( \begin{array}{cccccc} 
0 & 0 & 1 & 2 & 2 & 1 \\
1 & 0 & 0 & 1 & 2 & 2\\
2 & 1 & 0 & 0 & 1 & 2 \\
2 & 2 & 1 & 0 & 0 & 1 \\
1 & 2 & 2 & 1 & 0 & 0 \\
0 & 1 & 2 & 2 & 1 & 0
\end{array} \right),$$
has a $\RR_+^5$-factorization. Check that 
$$S_H = 
\left( \begin {array}{ccccc} 1&0&1&0&0\\\noalign{\medskip}1&0&0&0&1
\\\noalign{\medskip}0&0&0&1&2\\\noalign{\medskip}0&1&0&0&1\\\noalign{\medskip}0
&1&1&0&0\\\noalign{\medskip}0&0&2&1&0\end {array} \right) 
\left( \begin {array}{cccccc} 0&0&0&1&2&1\\\noalign{\medskip}1&2&1&0&0&0
\\\noalign{\medskip}0&0&1&1&0&0\\\noalign{\medskip}0&1&0&0&1&0
\\\noalign{\medskip}1&0&0&0&0&1\end {array} \right),$$
where we call the first matrix $A$ and the second matrix $B$. We may take the rows of $A$ 
as elements of $\RR^5_+$, and the columns of $B$ as elements of $\RR_+^5 = (\RR^5_+)^*$, and they provide us a $\RR_+^5$-factorization
of the slack matrix $S_H$, proving that this hexagon has a $\RR^5_+$-lift while the trivial polyhedral lift would have been to $\RR^6_+$.

We can construct the lift using the proof of the Theorem~\ref{thm:general_Yannakakis}. Note that 
$$H=\{(x_1,x_2) \in \RR^2 : \exists \,\,y \in \RR_+^5 \textrm{ s.t. } F x + B^T y = d\}.$$
Hence, the exact slice of $\RR_+^5$ that is mapped to the hexagon is simply
$$\{y \in \RR_+^5 : \exists \,\,x \in \RR^2 \textrm{ s.t. }  B^T y = d - F x\}.$$
By eliminating the $x$ variables in the system we get 
$$\{ y \in \RR^5_+ \,:\, y_1 + y_2 + y_3 + y_5 = 2, y_3 + y_4 + y_5 = 1 \},$$
and so we have a three dimensional slice of $\RR^5_+$ projecting down to $H$. This projection is visualized
in Figure \ref{fig:hexagon_proj}.

\begin{figure} 
\begin{center}
\includegraphics[scale=0.4]{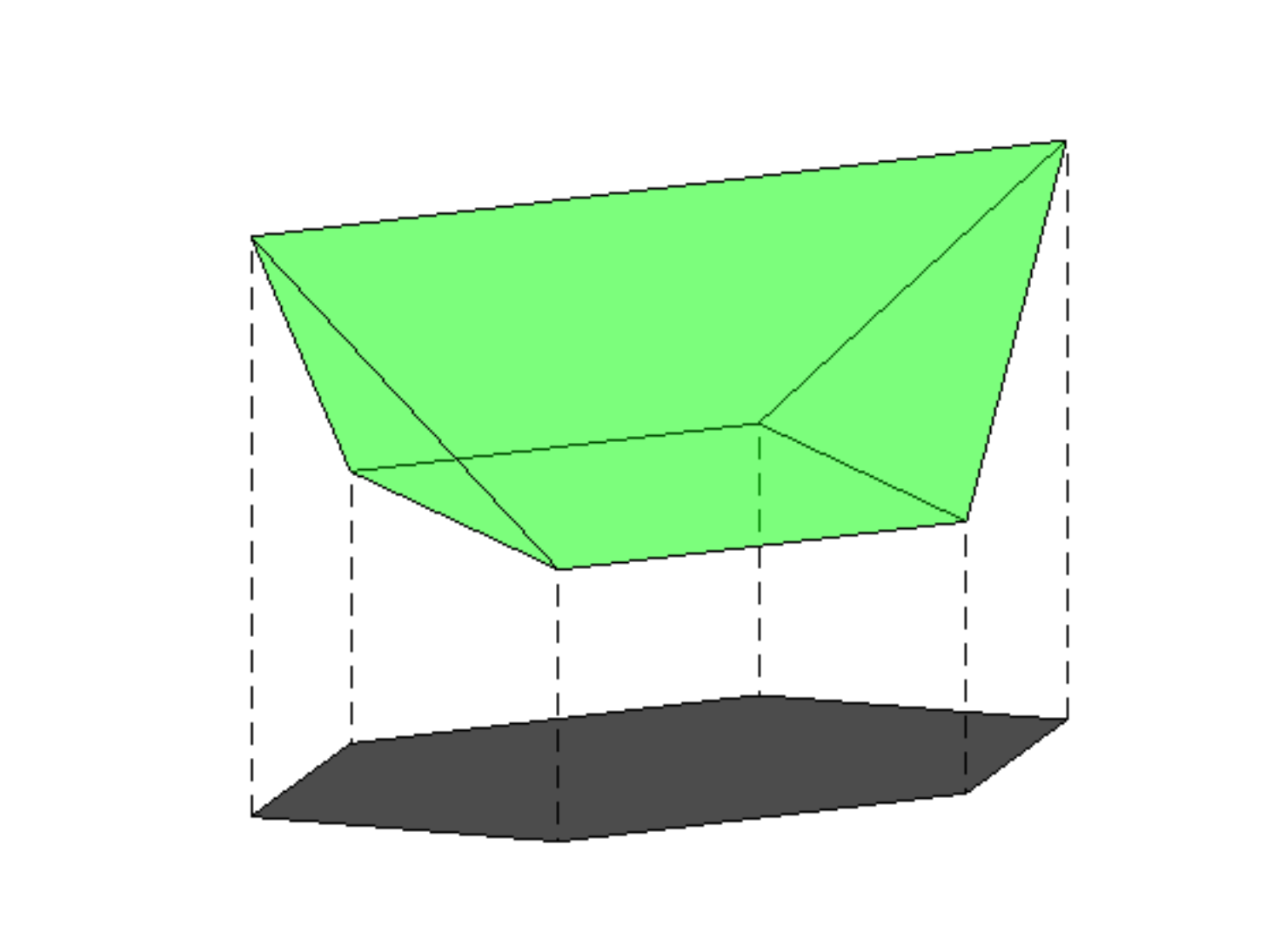}
\end{center}
\caption{A $\RR^5_+$-lift of the regular hexagon. \label{fig:hexagon_proj}
}
\end{figure}

The hexagon is a good example to see that the existence of lifts depends on more than the combinatorics of the polytope.
If instead of a regular hexagon we take the hexagon with vertices $(0,-1)$, $(1,-1)$, $(2,0)$, $(1,3)$, $(0,2)$ 
and $(-1,0)$, a valid slack matrix would be
$$ S := \left(
\begin{array}{cccccc}
 0 & 0 & 1 & 4 & 3 & 1 \\
 1 & 0 & 0 & 4 & 4 & 3 \\
 7 & 4 & 0 & 0 & 4 & 9 \\
 3 & 4 & 4 & 0 & 0 & 1 \\
 3 & 5 & 6 & 1 & 0 & 0 \\
 0 & 1 & 3 & 5 & 3 & 0
\end{array}
\right).$$
One can check that if a $6 \times 6$ matrix with the zero pattern of a slack matrix of a hexagon has a $\RR^5_+$-factorization, then it has a factorization with either the same zero pattern as the matrices $A$ and $B$ obtained
before, or the patterns given by applying a cyclic permutation to the rows of $A$ and the columns of $B$.
A simple algebraic computation then shows that the slack matrix $S$ above has no such decomposition hence this irregular hexagon has no $\RR_+^5$-lift. \qed
\end{example}

\begin{example}
In Example~\ref{ex:square} we saw a $\PSD^3$-lift of a square $P$. Up to scaling of columns by positive numbers, 
the slack matrix of $P$ is 
$$S_P = \begin{pmatrix} 0 & 0 & 1 & 1 \\ 0 & 1 & 1 & 0 \\1 & 1 & 0 & 0 \\ 1 & 0 & 0 & 1 \end{pmatrix} $$
where the rows are associated to the vertices $(1,1), \,(1,-1),\,(-1,-1), \,(-1,1)$ in that order, and the columns to the facets defined 
by the inequalities $$1-x_1 \geq 0, \,\,1 - x_2 \geq 0, \,\,1+x_1 \geq 0, \,\,1+x_2 \geq 0.$$

The matrix $S_P$ 
admits the following $\PSD^3$-factorization where the first four matrices are associated 
to the rows of $S_P$ and the next four matrices are associated to the columns of $S_P$:
$$ 
\begin{pmatrix} 1 & 1 & 1\\ 1 & 1 & 1 \\ 1 & 1 & 1 \end{pmatrix},
\begin{pmatrix} 1 & 1 & -1 \\ 1 & 1 & -1 \\ -1 & -1 & 1 \end{pmatrix},
\begin{pmatrix} 1 & -1 & -1 \\ -1 & 1 & 1 \\ -1 & 1 & 1 \end{pmatrix},
\begin{pmatrix} 1 & -1 & 1 \\ -1 & 1 & -1\\ 1 & -1 & 1 \end{pmatrix}
$$ 
$$ 
\frac{1}{4}\begin{pmatrix} 1 & -1 & 0 \\ -1 & 1 & 0 \\ 0 & 0 & 0 \end{pmatrix},
\frac{1}{4}\begin{pmatrix}  1 & 0 & -1 \\ 0 & 0 & 0 \\ -1 & 0 & 1 \end{pmatrix},
\frac{1}{4}\begin{pmatrix} 1 & 1 & 0 \\ 1 & 1 & 0 \\ 0 & 0 & 0 \end{pmatrix},
\frac{1}{4}\begin{pmatrix} 1 & 0 & 1 \\ 0 & 0 & 0 \\ 1 & 0 & 1 \end{pmatrix}.
$$
\qed
\end{example}

%%%%%%%%%%%%%%%%%%%%%%%%%%%%%%%%%%%%%%%%%%
\section{Positive Semidefinite Rank}
\label{sec:psd rank} 

From now on we focus on the special case of spectrahedral lifts of convex sets. Since the family of psd cones $\{ \PSD^k \,:\, k \in \NN\}$ is closed in the sense that any face of a member $\PSD^i$ in the family is isomorphic to $\PSD^j$ for some $j \leq i$, we can look at the smallest index $k$ for which a convex set $C$ admits a $\PSD^k$-lift.

\begin{definition} \label{def:psd rank}
The {\em psd rank} of a convex set $C \subset \RR^n$, denoted as $\psdrank(C)$ is the smallest positive integer  $k$ such that $C = \pi(\PSD^k \cap L)$ for some affine space $L$ and linear map $\pi$. If $C$ does not admit a psd lift, then define 
$\psdrank(C) = \infty$.
\end{definition} 

The following lemma is immediate from the previous sections and offers an explicit tool for establishing psd ranks. 

\begin{lemma} \label{lem:psdrank of matrices}
The psd rank of a convex set $C$ is the smallest $k$ for which the slack operator $S_C$ admits a $\PSD^k$-factorization. If $P$ is a polytope, then $\psdrank(P)$ is the smallest integer $k$ for which the slack matrix $S_P$ admits a $\PSD^k$-factorization.
\end{lemma}

Following Definition~\ref{def:K-factorization of nonnegative matrices}, for any nonnegative matrix $M \in \RR^{p \times q}_+$, one can define $\psdrank(M)$ to be the smallest integer $k$ such that $M$ admits a $\PSD^k$-factorization. The relationship between 
$\psdrank(M)$ and $\rank(M)$ is as follows:

\begin{align} \label{eq:bounds on psd rank}
\frac{1}{2} \left( \sqrt{1+8 \,\rank(M)} -1  \right) \leq \psdrank(M) \leq \min\{p,q\}.
\end{align}

For a proof, as well as a comprehensive comparison between psd rank and several other notions of rank 
of a nonnegative matrix, see \cite{psdranksurvey}.

The goal of this section is to describe the known bounds on psd ranks of convex sets. As might be expected, the best results we have are for polytopes.

\subsection{Polytopes}
In the case of polytopes, there is a simple lower bound on psd rank. The proof relies on the following 
technique to increase the psd rank of a matrix by one.

\begin{lemma} \cite[Proposition~2.6]{GRTpsdminimal} \label{lem:extending rank}
Suppose $M \in \RR^{p \times q}_+$ and $\psdrank (M) = k$. If $M$ is extended to 
$M' = \left( \begin{array}{cc} M & {\bf 0} \\ w & \alpha \end{array} \right)$ 
where $w \in \RR_+^q$, $\alpha > 0$ and ${\bf 0}$ is a column of zeros, then $\psdrank (M') = k+1$. Further, the factor 
associated to the last column of $M'$ in any $\PSD^{k+1}$-factorization of $M'$ has rank one.
\end{lemma}

\begin{theorem} \cite[Proposition~3.2]{GRTpsdminimal} \label{thm:lower bound on psd rank for polytopes}
If $P \subset \RR^n$ is a full-dimensional polytope, then the psd rank of $P$ is at least $n+1$. If 
$\psdrank (P) = n+1$, then {\em every} $\PSD^{n+1}$-factorization of the slack matrix of $P$ only uses rank one matrices as factors.
\end{theorem}

\begin{proof}
The proof is by induction on $n$. If $n=1$, then $P$ is a line segment and we may assume that its vertices are 
$p_1, p_2$ and facets are $F_1, F_2$ with $p_1 = F_2$ and $p_2 = F_1$. Hence its slack matrix is a $2 \times 2$ diagonal matrix with positive diagonal entries. It is not hard to see that  $\psdrank (S_P) = 2$ 
and any $\PSD^{2}$-factorization of it uses only matrices of rank one.

Assume the first statement in the theorem holds up to dimension $n-1$ and consider a polytope $P \subset \RR^n$ of dimension $n$. Let $F$ be a facet of $P$ with vertices $p_1, \ldots, p_s$, facets $f_1, \ldots, f_t$ and slack matrix $S_F$. Suppose $f_i$ corresponds to facet $F_i$ of $P$ for $i=1, \ldots, t$. By induction hypothesis, $\psdrank (F) = 
\psdrank (S_F) \geq n$. Let $p$ be a vertex of $P$ not in $F$ and assume that the top left $(s+1) \times (t+1)$ submatrix of $S_P$ is indexed by $p_1, \ldots, p_s, p$ in the rows and $F_1, \ldots, F_t, F$ in the columns. Then this submatrix of $S_P$, which we will call $S_F'$, has the form 
$$ S_F' = \left( \begin{array}{cc} S_F & {\bf 0} \\ * & \alpha \end{array} \right)$$ 
with $\alpha > 0$.  By Lemma~\ref{lem:extending rank}, the psd rank of $S_F'$ is at least $n+1$ since the psd rank of $S_F$ is at least $n$. Hence, $\psdrank (P) = \psdrank (S_P) \geq n+1$. 

Suppose there is now a $\PSD^{n+1}$-factorization of $S_P$ and therefore of $S_F'$. By 
Lemma~\ref{lem:extending rank} the factor 
corresponding to the facet $F$ has rank one. Repeating the procedure for all facets $F$ and all submatrices $S_F'$ we get that all factors corresponding to the facets of $P$ in this $\PSD^{n+1}$-factorization of $S_P$ must have rank one. To prove that all factors indexed by the vertices of $P$ also have rank one, we use the fact that the transpose of a slack matrix of $P$ is (up to row scaling) a slack matrix of the polar polytope $P^\circ$, concluding the proof.
\end{proof}

For an $n$-dimensional polytope $P \subset \RR^n$, it is well-known that $\rank(S_P) = n+1$, see for instance \cite[Lemma~3.1]{GRTpsdminimal}. 
Therefore, Theorem~\ref{thm:lower bound on psd rank for polytopes} implies that for a
slack matrix $S_P$ of a polytope $P$ we have a simple relationship between rank and psd rank, namely 
$\rank(S_P) \leq \psdrank(P)$, as compared to \eqref{eq:bounds on psd rank}. 
From \eqref{eq:bounds on psd rank} we also have that for a polytope $P$ with $v$ vertices and $f$ facets, \
$\psdrank(P) \leq \min\{v,f\}$. 
In general, it is not possible to bound the psd rank of nonnegative matrices, even slack matrices, by a function in the rank of the matrix. For instance, all slack matrices of polygons have rank three. However, we will see as a consequence of  the results in the next subsection that the psd rank of an $n$-gon grows with $n$.

In the next section we will see that the lower bound in Theorem~\ref{thm:lower bound on psd rank for polytopes} can be tight for several interesting classes of polytopes. Such polytopes include some $0/1$-polytopes.
However, Bri\"et, Dadush and Pokutta showed that not all $0/1$-polytopes can have small psd rank.

\begin{theorem} \cite{BrietDadushPokutta}
For any $n \in \ZZ_+$, there exists $U \subset \{0,1\}^n$ such that 
$$ \psdrank(\conv(U)) = \Omega\left( \frac{2^{\frac{n}{4}}}{(n \log n)^\frac{1}{4}} \right).$$
\end{theorem}

Despite the above result, it is not easy to find explicit polytopes with high psd rank. 
The most striking results we have so far are the following by Lee, Raghavendra and Steurer, which provide super polynomial lower bounds on the psd 
rank of specific families of $0/1$-polytopes.  

\begin{theorem} \cite{LeeRaghavendraSteurer}
The cut, TSP, and stable set polytopes of $n$-vertex graphs have psd rank at least $2^{n^\delta}$, for some constant $\delta > 0$. 
\end{theorem}

We saw the stable set polytope of an $n$-vertex graph before. The cut and TSP polytopes are other examples of polytopes that come from graph optimization problems. The TSP (traveling salesman problem) is the problem of finding a tour through all vertices of the $n$-vertex complete graph that minimizes a linear objective function. Each tour can be represented as a $0/1$-vector in $\{0,1\}^{n \choose 2}$ and the TSP polytope is the convex hull of all these tour vectors.

\subsection{General convex sets}
We now examine lower bounds on the psd rank of an arbitrary convex set $C \subset \RR^n$.  
The following elegant lower bound was established by 
Fawzi and Safey El Din.

\begin{theorem} \cite{Fawzi-SafeyElDin} \label{thm:lower bound on psd rank of convex sets}
Suppose $C \subset \RR^n$ is a convex set and $d$ is the minimum degree of a polynomial with real coefficients that vanishes on the boundary of 
$C^\circ$. Then $\psdrank(C) \geq \sqrt{\log d}$.
\end{theorem}

The {\em algebraic degree} of a convex set $C$ is the smallest degree of a polynomial with real coefficients that vanishes on the boundary of $C$. Suppose $P$ is a polytope with $v$ vertices and the origin in its interior. Then $P^\circ$ has $v$ facets each corresponding to a linear polynomial $l_i$ that vanishes on the facet. The polynomial 
$p := \pi_{i=1}^v l_i$ vanishes on the boundary of $P^\circ$ and has degree $v$. In fact, the algebraic degree of $P^\circ$ is $v$. Hence by Theorem~\ref{thm:lower bound on psd rank of convex sets}, $\psdrank(P) \geq \sqrt{\log v}$.  
This result is analogous to an observation of Goemans \cite{Goemans} that any polyhedral lift $Q$ of $P$ has at least $\log v$ facets. The reason is that every vertex in $P$ is the projection of a face of $Q$ which in turn is the intersection of some set of facets of $Q$. Therefore, 
$$ v \leq \# \textup{ faces of } Q \leq 2^{\# \textup{ facets of } Q}.$$  
Even for polytopes there are likely further factors from combinatorics and topology that can provide stronger lower bounds on psd rank.

The lower bound in Theorem~\ref{thm:lower bound on psd rank of convex sets} is very explicit and simple, but it does not involve $n$. We now exhibit a simple lower bound that does.

\begin{proposition} \label{prop:lower bound depending on n}
 Let $C \subset \RR^n$ be an $n$-dimensional convex body. Then \\
$\psdrank(C) = \Omega(\sqrt{n})$.
\end{proposition}

\begin{proof} 
Suppose $\psdrank(C) = k$. Then there exists maps $A \,:\, \textup{ext}(C) \rightarrow \PSD^k$ and $B \,:\, \textup{ext}(C^\circ) \rightarrow \PSD^k$ such that for all $(x,y) \in \textup{ext}(C) \times \textup{ext}(C^\circ)$, 
\begin{align} \label{eq:slack operator psd rank bound}
S_C((x,y)) =  1 - \langle x, y \rangle = (1, x^\top) \cdot \begin{pmatrix} 1 \\ -y \end{pmatrix} = \textup{trace}(A(x)B(y)).
\end{align}

Define $\rank(S_C)$ to be the minimum $l$ such that $S_C((x,y)) = a_x^\top b_y$ for $a_x,b_y \in \RR^l$. 
Equality of the first and third expressions in \eqref{eq:slack operator psd rank bound} implies that $\rank(S_C) \leq n+1$.
Now consider $n+1$ affinely independent extreme points $x_1, \ldots, x_{n+1}$ of $C$ and $n+1$ affinely independent extreme points 
$y_1, \ldots, y_{n+1}$ of $C^\circ$. Then the values of $S_C$ restricted to $(x,y)$ as $x$ and $y$ vary in these chosen sets are the entries of the matrix
$$ \begin{pmatrix} 1 &  x_1^\top \\ \vdots \\ 1 & x_{n+1}^\top \end{pmatrix}
\begin{pmatrix} 1  &  \cdots & 1  \\
-y_1 &  \cdots & -y_{n+1} \end{pmatrix}$$ 
which has rank $n+1$. Therefore, $\rank(S_C) = n+1$.
Equality of the first and last expressions in \eqref{eq:slack operator psd rank bound} implies that the first inequality in \eqref{eq:bounds on psd rank} holds with $M$ replaced by $S_C$ via the same proof, see \cite[Proposition 4]{GPT-lifts}. In other words,   
$  \frac{1}{2} \left( \sqrt{1+8 (n+1)} -1  \right) \leq \psdrank(S_C) = \psdrank(C),$ 
and we get the result.
\end{proof}

\begin{example} \label{ex:ball continued}
The spectrahedral lift of $B_{n^2}$ in Example~\ref{ex:ball} is optimal, and $\psdrank (B_{n^2}) = \Theta(n)$. 
\qed
\end{example}

The lower bounds in Theorem~\ref{thm:lower bound on psd rank of convex sets} and Proposition~\ref{prop:lower bound depending on n} depend solely on the algebraic degree of $C^\circ$
and $n$ respectively. A question of interest is how the bound might jointly depend on both these parameters?

While the lower bounds in Theorems~\ref{thm:lower bound on psd rank for polytopes}, \ref{thm:lower bound on psd rank of convex sets} and Proposition~\ref{prop:lower bound depending on n} can be tight, we do not have much understanding of the psd ranks of specific polytopes or convex sets except in a few cases. 
For example, Theorem~\ref{thm:lower bound on psd rank of convex sets} implies that the psd rank of polygons must grow to infinity as the number of vertices grows to infinity. However, we do not know if the psd rank of polygons is monotone in the number of vertices. 

%%%%%%%%%%%%%%%%%%%%%%%%%%%%%%%%%%%%%%%%%%%%%%%%%%%%%%%%%%%%%%
\section{Psd-Minimal Polytopes}
\label{sec:psd minimality}

Recall from Theorem~\ref{thm:lower bound on psd rank for polytopes} that the psd rank of an $n$-dimensional polytope is at least $n+1$. In this section we study those polytopes whose psd rank is exactly this lower bound. Such polytopes are said to be {\em psd-minimal}. The key to understanding psd-minimality is another notion of rank of a nonnegative matrix.

\begin{definition} \label{def:Hadamard square root} 
A {\em Hadamard square root} of a nonnegative real matrix $M$, denoted as $\sqrt{M}$,  is any matrix 
whose $(i,j)$-entry is a square root (positive or negative) of the $(i,j)$-entry of $M$.  
\end{definition}

Let $\sqrtrank (M) := \textup{min} \{ \rank(\sqrt{M}) \}$ be the minimum rank of a Hadamard square root of a nonnegative matrix $M$. We recall the basic connection between the psd rank of a nonnegative matrix $M$ and $\sqrtrank (M)$ shown in \cite[Proposition 2.2]{GRTpsdminimal}.

\begin{proposition} \label{prop:psd rank and Hadamard square roots}
If $M$ is a nonnegative matrix, then $\psdrank (M) \leq \sqrtrank (M)$.
In particular, the psd rank of a $0/1$ matrix is at most the rank of the matrix.
\end{proposition}

\begin{proof} Let $\sqrt{M}$ be a Hadamard square root of $M \in \RR^{p \times q}_+$ of rank $r$. Then there exist vectors 
$a_1, \ldots, a_p, b_1, \ldots, b_q \in \RR^{r}$ such that $(\sqrt{M})_{ij} = \langle a_i, b_j \rangle$. Therefore,  
$M_{ij} = \langle a_i, b_j \rangle^2 = \langle a_i a_i^T, b_j b_j^T \rangle$ where the second inner product is the trace inner product for symmetric matrices defined earlier.  Hence, $\psdrank (M) \leq r$.
\end{proof}

Even though $\sqrtrank(M)$ is only an upper bound on $\psdrank(M)$, we cannot find $\PSD^k$-factorizations of $M$ with only rank one factors if $k < \sqrtrank(M)$.

\begin{lemma} \cite[Lemma 2.4]{GRTpsdminimal} \label{lem:rank one factors}
The smallest $k$ for which a nonnegative real matrix $M$ admits a $\PSD^k$-factorization in which all factors are matrices of rank one is $k = \sqrtrank(M)$.
\end{lemma}

\begin{proof} If $k = \sqrtrank (M)$, then there is a Hadamard square root of $M \in \RR^{p \times q}_+$ of rank $k$ and the proof of Proposition~\ref{prop:psd rank and Hadamard square roots} gives a $\PSD^k$-factorization of $M$ in which all factors have rank one.
On the other hand, if there exist $a_1a_1^T, \ldots, a_pa_p^T, b_1b_1^T, \ldots, b_qb_q^T \in \PSD^k$ such that $M_{ij}= \langle a_ia_i^T, b_jb_j^T \rangle = \langle a_i, b_j \rangle^2$, then the matrix with $(i,j)$-entry $\langle a_i, b_j \rangle$ is a Hadamard square root of $M$ of rank at most $k$. 
\end{proof} 

This brings us to a characterization of psd-minimal polytopes.

\begin{theorem} \label{thm:psdrank n+1} 
If $P \subset \RR^n$ is a full-dimensional polytope, then  $\psdrank (P) = n+1$ if and only if 
$\sqrtrank (S_P) = n+1$.
\end{theorem}

\begin{proof}
By Proposition~\ref{prop:psd rank and Hadamard square roots}, $\psdrank (P) \leq \sqrtrank (S_P)$. Therefore, if $\sqrtrank 
(S_P)  = n+1$, then by Theorem~\ref{thm:lower bound on psd rank for polytopes}, the psd rank of $P$ is exactly $n+1$.

Conversely, suppose $\psdrank (P) = n+1$. Then there exists a $\PSD^{n+1}$-factorization of $S_P$ which, by Theorem~\ref{thm:lower bound on psd rank for polytopes}, has all factors of rank one. Thus, by Lemma~\ref{lem:rank one factors}, we have $\sqrtrank (S_P) \leq n+1$.  Since $\sqrtrank$ is bounded below by $\psdrank$, we must have $\sqrtrank (S_P) = n+1$.
\end{proof}

Our next goal is to find psd-minimal polytopes. Recall that two polytopes $P$ and $Q$ are 
{\em combinatorially equivalent} if they have the same vertex-facet incidence structure. 
In this section we describe a simple algebraic obstruction to psd-minimality based on the combinatorics of a given polytope, therefore providing an obstruction for all polytopes in the given combinatorial class. 
Our main tool is a symbolic version of the slack matrix of a polytope.

\begin{definition} The {\em symbolic slack matrix} of a $d$-polytope $P$ is the matrix, $S_P(x)$, obtained by replacing all positive entries in the slack matrix $S_P$ of $P$ with distinct variables $x_1, \ldots, x_t$.
\end{definition}

Note that two $d$-polytopes $P$ and $Q$ are in the same combinatorial class if and only if
$S_P(x) = S_Q(x)$ up to permutations of rows and columns, and names of variables. Call a polynomial $f \in \RR[x_1,\ldots,x_t]$ a {\em monomial} if it is of the form $f = \pm x^a$ where $x^a = x_1^{a_1} \cdots x_t^{a_t}$ and $a=(a_1,\ldots,a_t) \in \NN^t$.  We refer to a  sum of two distinct monomials as a {\em binomial} and to the sum of three distinct monomials as a {\em trinomial}. 
 This differs from the usual terminology that allows nontrivial coefficients.
 
\begin{lemma}[Trinomial Obstruction Lemma] \label{lem:trinomial obstruction}
Suppose the symbolic slack matrix $S_P(x)$ of an $n$-polytope $P$ has a $(n+2)$-minor 
that is a trinomial. Then no polytope in the combinatorial class of $P$ can be psd-minimal.
\end{lemma}

\begin{proof}  Suppose $Q$ is psd-minimal and combinatorially equivalent to $P$.
Hence, we can assume that $S_P(x)$ equals $S_Q(x)$. By Theorem~\ref{thm:psdrank n+1} there is some $u=(u_1,\ldots, u_t)\in\RR^t$, with no coordinate equal to zero, such that $S_Q=S_P(u^2_1,\ldots,u^2_t)$ and $\rank(S_P(u)) =n+1$. Since $S_Q$ is the slack matrix of an $n$-polytope, we have  
$$
\rank(S_P(u^2_1,\ldots,u^2_t))=n+1 = \rank(S_P(u_1,\ldots, u_t)).
$$

Now suppose $D(x)$ is a trinomial $(n+2)$-minor of $S_P(x)$. Up to sign, $D(x)$ has the form $x^a + x^b + x^c$ or $x^a - x^b + x^c$ for some $a, b, c\in \NN^t$. In either case, 
it is not possible for $D(u^2_1,\ldots,u^2_t)=D(u_1,\ldots,u_t)=0$.
\end{proof}

\subsection{Psd-minimal polytopes of dimension up to four}

\begin{proposition} \cite[Theorem~4.7]{GRTpsdminimal} \label{prop:psd minimal in the plane}
The psd-minimal polygons are precisely all triangles and quadrilaterals.
\end{proposition}

\begin{proof}
Let $P$ be an $n$-gon where $n > 4$. Then $S_P(x)$ has a submatrix of the form 
$$ \begin{bmatrix} 
0 & x_1 & x_2 & x_3 \\ 0 & 0 & x_4 & x_5 \\ x_6 & 0 & 0 & x_7 \\ x_8 & x_9 & 0 & 0 
\end{bmatrix},$$
whose determinant is $x_1x_4x_7x_8 -  x_2x_5x_6x_9 + x_3x_4x_6x_9 $ up to sign. 
By Lemma~\ref{lem:trinomial obstruction}, no $n$-gon with $n > 4$ can be psd-minimal.

Since all triangles are projectively equivalent, by verifying the psd-minimality of one, they are 
all seen to be psd-minimal. Similarly, for quadrilaterals.
\end{proof}

Lemma~\ref{lem:trinomial obstruction} can also be used to 
classify up to combinatorial equivalence all $3$-polytopes that are psd-minimal. 
Using Proposition~\ref{prop:psd minimal in the plane}, together with the fact that faces of psd-minimal polytopes are also psd-minimal, and the invariance of psd rank under polarity, we get that that any $3$-polytope $P$ with a vertex of degree larger than four, or a facet that is an $n$-gon where $n > 4$, 
cannot be  psd-minimal.  

\begin{lemma} \label{lem:no square facet incident to a degree 4 vertex}
If $P$ is a $3$-polytope with a vertex of degree four and a quadrilateral 
facet incident to this vertex, then $S_P(x)$ contains a trinomial $5$-minor.
\end{lemma}

\begin{proof} Let $v$ be the vertex of degree four incident to facets $F_1,F_2,F_3,F_4$ such that $[v_1,v]=F_1\cap F_2$, $[v_2,v]=F_2\cap F_3$, $[v_3, v]=F_3\cap F_4$ and $F_4\cap F_1$ are edges of $P$, where $v_1$, $v_2$ and $v_3$ are vertices of $P$. 

Suppose $F_4$ is quadrilateral. Then $F_4$ has a vertex $v_4$ that is different from, and non-adjacent to, $v$. Therefore, $v_4$ does not lie on  $F_1$, $F_2$ or  $F_3$.  Consider the $5 \times 5$ submatrix of $S_P(x)$ with rows indexed by $v,v_1,v_2,v_3,v_4$ and columns by 
$F_1,F_2,F_3,F_4,F$ where $F$ is a facet not containing $v$. This matrix has the form 
$$
\begin{bmatrix} 0 & 0 & 0 & 0 & x_1\\ 0 & 0 & x_2 & x_3 & *\\ x_4 & 0 & 0 & x_5 & * \\ x_6 & x_7 & 0 & 0 & * \\ x_8 & x_9 & x_{10} & 0 & * \\ 
\end{bmatrix},
$$
and its determinant is a trinomial. 
\end{proof}

\begin{proposition} \label{prop:combinatorial classification in 3-space}
The psd-minimal $3$-polytopes are combinatorially equivalent to simplices, quadrilateral pyramids, bisimplices, octahedra or 
their duals.
\end{proposition}

\begin{proof}
Suppose $P$ is a psd-minimal $3$-polytope. If $P$ contains only vertices of degree three and triangular facets, then $P$ is a simplex.

For all remaining cases, $P$ must have a vertex of degree four or a quadrilateral facet.  Since psd rank is preserved under polarity, we may assume that $P$ has a vertex $u$ of degree four.  By Lemma~\ref{lem:no square facet incident to a degree 4 vertex}, the neighborhood of $u$ looks as follows.
\begin{center}
\begin{tikzpicture}
  [scale=.2,auto=center]
  \coordinate (p) at (6,6);
  \coordinate (p1) at (1,11);
  \coordinate (p2) at (1,1);
  \coordinate (p3) at (11,1);
  \coordinate (p4) at (11,11);
  \node[label=$u$] at (p) {};
  \node[label=$v_1$] at ($(p1)-(1.2,1.7)$) {};
  \node[label=$v_2$] at ($(p2)-(1.2,1.3)$) {};
  \node[label=$v_3$] at ($(p3)+(1.2,-1.3)$) {};
  \node[label=$v_4$] at ($(p4)+(1.2,-1.7)$)  {};
  \foreach \from/\to in {p/p1,p/p2,p/p3,p/p4,p3/p4,p4/p1,p1/p2,p2/p3}
    \draw (\from) -- (\to);
\end{tikzpicture}
\end{center}

Suppose $P$ has five vertices. If all edges of $P$ are in the picture, i.e. the picture is a Schlegel diagram of $P$, then $P$ is a quadrilateral pyramid. Otherwise $P$ has one more edge, and this edge is $[v_1,v_3]$ or $[v_2, v_4]$, yielding a bisimplex in either case. 

If $P$ has more than five vertices, then we may assume that $P$ has a vertex $v$ that is a 
neighbor of $v_1$ different from $u$, $v_2$, $v_4$.  Then $v_1$ is a degree four vertex and thus, by Lemma~\ref{lem:no square facet incident to a degree 4 vertex}, all facets of $P$ containing $v_1$ are triangles.  
This implies that $v$ is a neighbor of $v_2$ and $v_4$. Applying the same logic to either $v_2$ or $v_4$, we get that $v$ is also a neighbor of $v_3$. Since all these vertices now have degree four, there could be no further vertices in $P$, and so $P$ is an octahedron.
Hence $P$ is combinatorially 
equal to, or dual to, one of the polytopes seen so far.
\end{proof}

Call an octahedron in $\RR^3$, {\em biplanar}, if there are two distinct planes each containing four vertices of the octahedron. The  complete classification of psd-minimal $3$-polytopes  is as follows.

\begin{theorem} \cite[Theorem~4.11]{GRTpsdminimal} \label{thm:classification in 3-space}
The psd-minimal $3$-polytopes are precisely simplices,  quadrilateral pyramids, bisimplices, biplanar octahedra and their polars.
\end{theorem}

In dimension four, the classification of psd-minimal polytopes becomes quite complicated. The full list consists of $31$ combinatorial classes of polytopes including the $11$ known projectively unique polytopes in $\RR^4$. These $11$ are {\em combinatorially psd-minimal}, meaning that all polytopes in each of their combinatorial classes are psd-minimal. For the remaining 20 classes, there are non-trivial conditions on psd-minimality. We refer the reader to \cite{GPRT} for the result in $\RR^4$.

Beyond $\RR^4$, a classification of all psd-minimal polytopes looks to be cumbersome. 
On the other hand, there are families of polytopes of increasing dimension that are all psd-minimal. 
A polytope $P \subset \RR^n$ is $2$-{\em level} if for every facet of $P$, 
all vertices of $P$ are either on this facet or on a single other parallel translate of the affine span of this facet. Examples of $2$-level polytopes include simplices, regular hypercubes, regular cross-polytopes, and hypersimplices. All $2$-level polytopes are psd-minimal, but not conversely. For example, the regular bisimplex in $\RR^3$ is psd-minimal but not $2$-level. Recall from Example~\ref{ex:stable set polytope} that the stable set polytopes of perfect graphs are psd-minimal. In fact, they are also $2$-level and it was shown in \cite[Corollary~4.11]{GPT-thetabodies} that all down-closed $0/1$-polytopes that are $2$-level are in fact stable set polytopes of perfect graphs. On the other hand, \cite[Theorem 9]{GPT-lifts} shows that $\textup{STAB}(G)$ is not psd-minimal if $G$ is not perfect.

%%%%%%%%%%%%%%%%%%%%%%%%%%%%%%%%%%%%%%%%%%%%%%%%%%%%%%%%%%%
\section{Spectrahedral lifts and sum of squares polynomials}
\label{sec:sos}

We now look at a systematic technique that creates a sequence of nested outer approximations of the 
convex hull of an algebraic set. These approximations come from projections of spectrahedra and are called {\em theta bodies}. In many cases, the theta 
body at the $k$th step will equal the closure of the convex hull of the algebraic set and hence the spectrahedron that it 
was a projection of, is a lift of this convex set. We examine how this type of lift fits into our general picture. 

Let $I = \langle p_1, \ldots, p_s \rangle 
\subset \RR[x] := \RR[x_1, \ldots, x_n]$ be a polynomial ideal and let $\VV_\RR(I) \subset \RR^n$ be the real points in its variety. Then the closure 
of the convex hull of $\VV_\RR(I)$,  $C := \overline{\conv(\VV_\RR(I))}$, is a closed convex semialgebraic set. Since we are only interested in the convex hull of 
$\VV_\RR(I)$, and the convex hull is defined by its extreme points, we may assume without loss of generality that $I$ is the largest ideal that vanishes on the extreme points of $C$.

Recall that $C$ is the intersection of all half spaces containing $\VV_\RR(I)$. Each half space is expressed 
as $l(x) \geq 0$ for some linear polynomial $l \in \RR[x]$ that is nonnegative on $\VV_\RR(I)$. A linear polynomial $l$ is nonnegative on 
$\VV_\RR(I)$ if there exists polynomials $h_i \in \RR[x]$ such that $l - \sum h_i^2 \in I$. In this case we say that $l$ is a {\em sum of squares (sos)} mod $I$, and if the degree of each $h_i$ is at most $k$, then we say that $l$ is $k$-{\em sos mod} $I$. Define the $k$th theta body of $I$ to be the set 
$$\textup{TH}_k(I) := \left\{ x \in \RR^n \,:\, l(x) \geq 0 \,\,\forall \,\, l \textup{ linear and } k\textup{-sos mod } I \right\}.$$
  
Note that all theta bodies are closed convex semialgebraic sets and they form a series of nested outer approximations of $C$ 
since $$\textup{TH}_i(I) \supseteq \textup{TH}_{i+1}(I) \supseteq C \,\,\textup{ for all } \,\,i \geq 1.$$
We say that $I$ is $\textup{TH}_k${\em-exact} if $\textup{TH}_k(I) = C$. 
The terminology is inspired by Lov\'asz's theta body $\textup{TH}(G)$ from Example~\ref{ex:stable set polytope} which is precisely $\textup{TH}_1(I_G)$ of the ideal 
$$I_G = \langle x_i^2 - x_i, \,\,\forall i=1,\ldots, n \rangle + \langle x_ix_j, \,\,\forall \{i,j\} \in E(G) \rangle.$$
In our terminology, $I_G$ is $\textup{TH}_1$-exact when $G$ is a perfect graph.

\begin{figure} 
\includegraphics[scale=0.23]{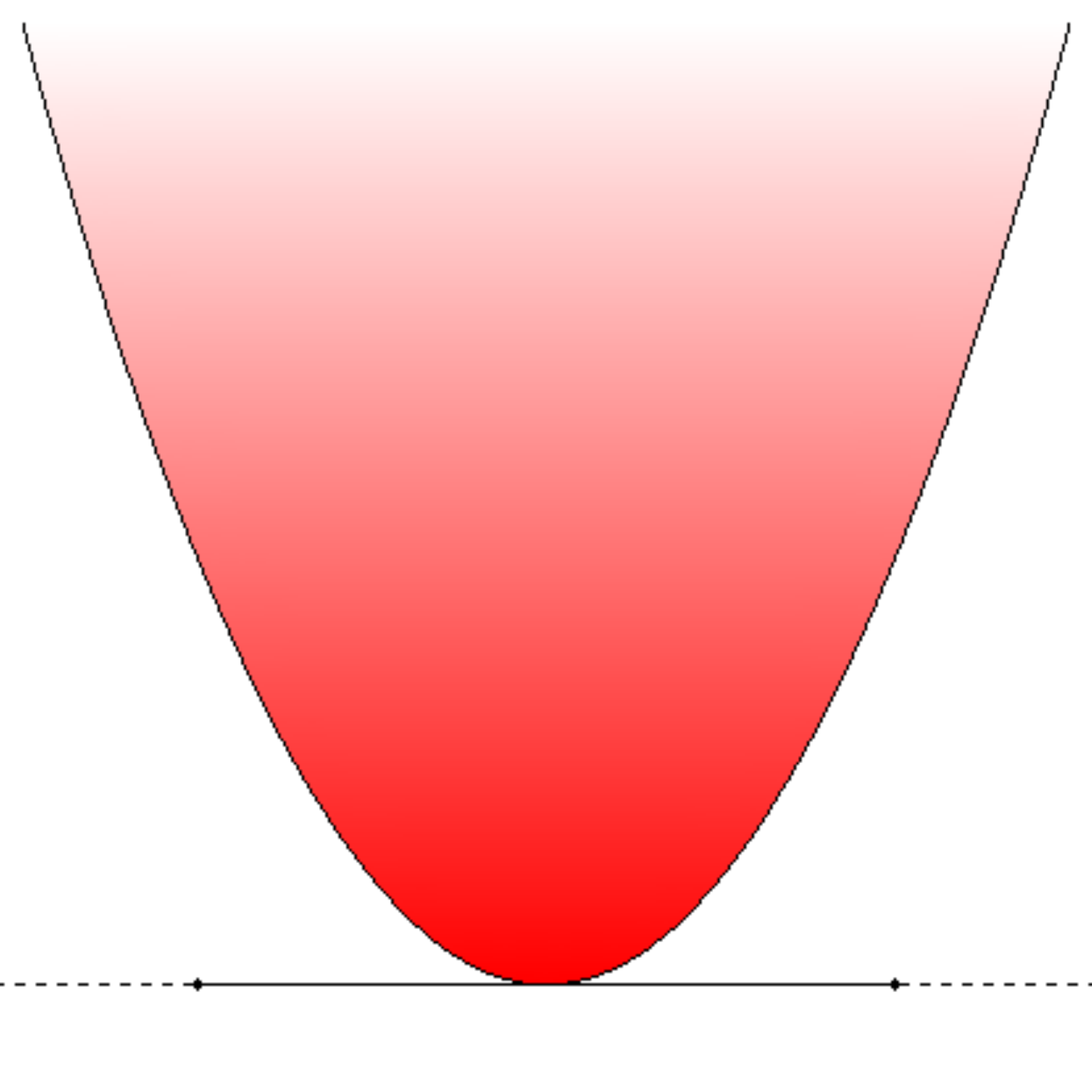}
\includegraphics[scale=0.23]{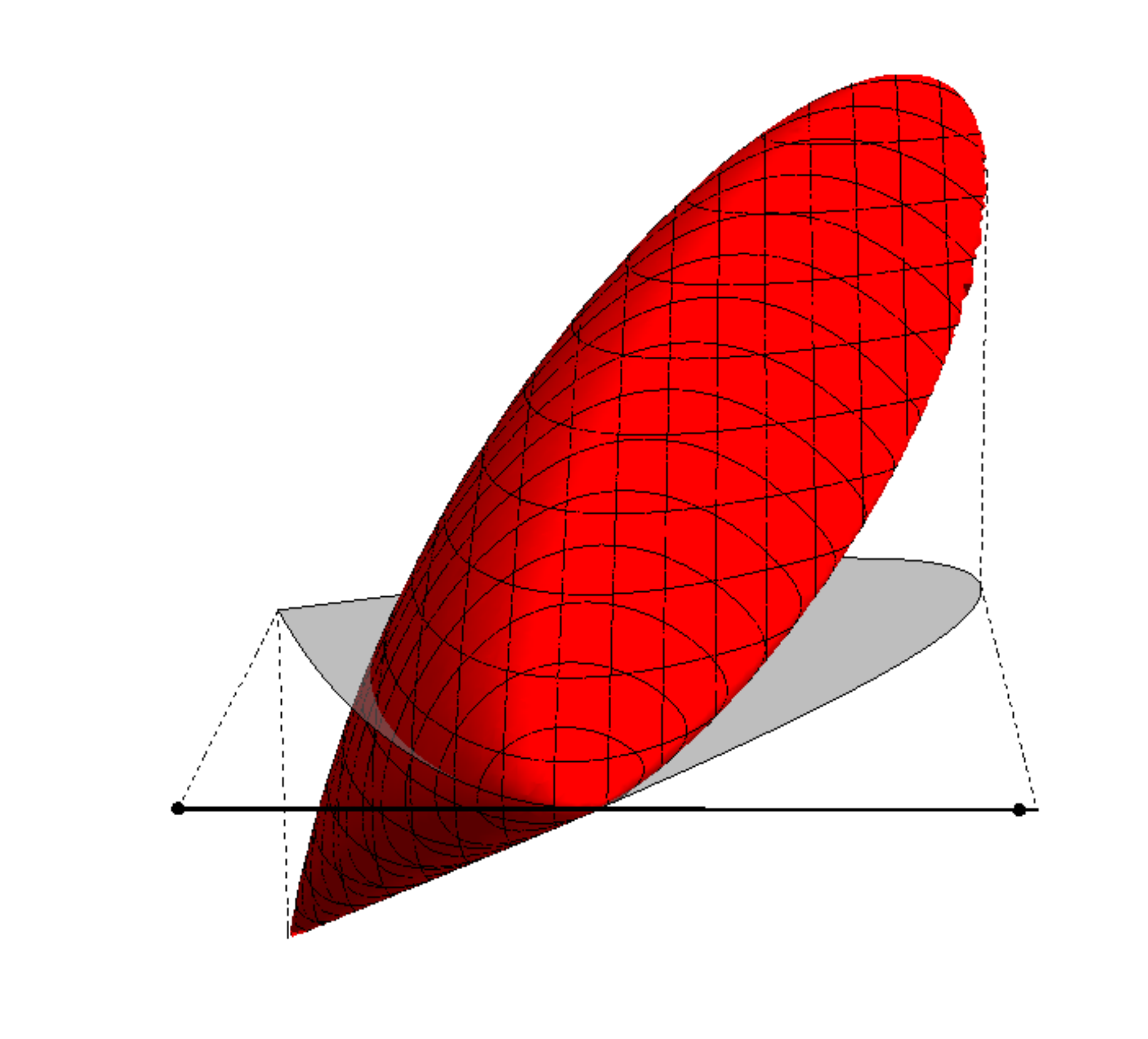}
\includegraphics[scale=0.23]{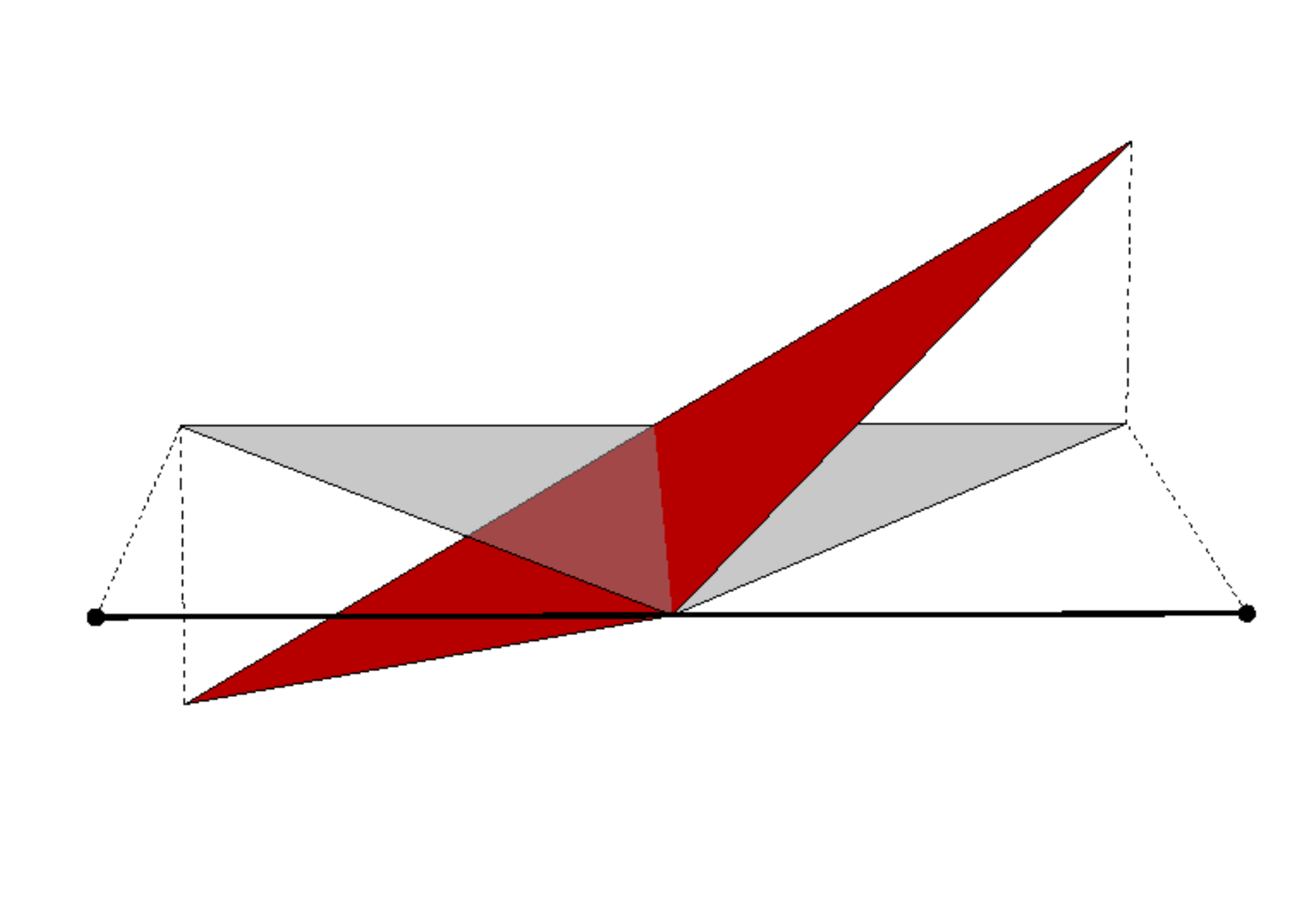}
\caption{The theta bodies of $I = \langle
  (x+1)x(x-1)^2 \rangle$ and their spectrahedral lifts. The first theta body is the entire real line, the second is slightly larger than 
  $[-1,1]$ and the third is exactly $[-1,1]$. \label{fig:theta bodies}
 }
\end{figure}

Theta bodies of a general polynomial ideal $I \subset \RR[x]$ were defined in \cite{GPT-thetabodies}, and it was shown 
there that $I$ is $\textup{TH}_k$-exact if and only if $C$ admits a specific type of spectrahedral lift.   This lift has 
size equal to the number of monomials in $\RR[x]$ of degree at most $k$. Let $[x]_k$ denote the vector of all 
monomials of degree at most $k$ in $\RR[x]$. When $\textup{TH}_k(I) = C$, Theorem~\ref{thm:general_Yannakakis} promises two maps 
$A$ and $B$ that factorize the slack operator of 
$C$. These operators are very special.

\begin{theorem} \cite[Theorem~11]{GPT-lifts} \label{thm:theta body factorizations}
The slack operator of $C = \overline{\conv(\VV_\RR(I))}$ has a factorization in which $A(x) = [x]_k[x]_k^\top$ 
if and only if $C= \textup{TH}_k(I)$. Further, the map 
$B$ sends each linear functional $l(x)$ corresponding to an extreme point of the polar of $C$ to a 
psd matrix $Q_l$ such that $l(x) - x^\top Q_l x \in I$ certifying that $l(x)$ is nonnegative on $\VV_\RR(I)$.
\end{theorem}

In fact, each theta body is the projection of a spectrahedron. Figure~\ref{fig:theta bodies} shows the theta bodies and their spectrahedral lifts of the ideal $I =  \langle (x+1)x(x-1)^2 \rangle$. In this case, $C = [-1,1] \subset \RR$.

While theta bodies offer a systematic method to sometimes construct a spectrahedral lift of $C$,
they may not offer the most efficient lift of this set. So an immediate question is whether there might be radically different 
types of spectrahedral lifts for $C$. Since the projection of a spectrahedron is necessarily convex and 
semialgebraic,  a set $C$ can have a spectrahedral lift only if it is convex and semialgebraic. So a second question is whether 
every convex semialgebraic set has a spectrahedral lift. This question gained prominence from \cite{Nemirovski}, 
and Helton and Nie showed that indeed a compact convex semialgebraic set has a spectrahedral lift if its boundary is sufficiently smooth and has positive curvature. They then conjectured that every convex semialgebraic set has a spectrahedral lift, see \cite{HeltonNie2009} and \cite{HeltonNie2010}. This conjecture was very recently disproved by Scheiderer who exhibited many explicit counter-examples \cite{Scheiderer-HeltonNie}. 
All these sets therefore have infinite psd rank. 
%Scheiderer's results also prove that the theta body construction of a spectrahedral lift in the setting we have is essentially universal. His results apply more generally to convex hulls of semialgebraic sets, as opposed to algebraic sets. 

Recall that a morphism $\phi \,:\, X \rightarrow Y$ between two affine real varieties creates a ring homomorphism 
$\phi^\ast \,:\, \RR[Y] \rightarrow \RR[X]$ between their coordinate rings. By a real variety we mean a variety defined by polynomials with real coefficients. Let $X_\RR$ denote the $\RR$-points of $X$.

\begin{theorem} \cite[Theorem 3.14]{Scheiderer-HeltonNie} \label{thm:Scheiderer's key theorem}
Let $S \subset \RR^n$ be a semialgebraic set and let $C$ be the closure of its convex hull. Then $C$ has a spectrahedral lift 
if and only if there is a morphism $\phi \,:\, X \rightarrow \mathbb{A}^n$ of affine real varieties and a finite-dimensional 
$\RR$-linear subspace $U$ in the coordinate ring $\RR[X]$ such that 
\begin{enumerate}
\item $S \subset \phi(X_\RR)$,
\item for every linear polynomial $l \in \RR[x]$ that is nonnegative on $S$, the element $\phi^\ast(l)$ of 
$\RR[X]$ is a sum of squares of elements in $U$.
\end{enumerate}
\end{theorem}

This theorem offers a set of necessary and sufficient conditions for the existence of a spectrahedral lift of the convex hull of a semialgebraic set by working through an intermediate variety $X$. The setting is more general than that in Theorem~\ref{thm:theta body factorizations} where we only considered convex hulls of algebraic sets. Regardless, the spirit of condition (2) is that 
the theta body method (or more generally, Lasserre's method \cite{Lasserre}) is essentially universal with the 
subspace $U \subseteq \RR[X]$ playing the role of degree bounds on the 
sos nonnegativity certificates that were required for $\textup{TH}_k$-exactness. Theorem~\ref{thm:Scheiderer's key theorem} 
provides counterexamples to the Helton-Nie conjecture. 

\begin{theorem} \cite[Theorem 4.23]{Scheiderer-HeltonNie}
Let $S \subset \RR^n$ be any semialgebraic set with $\dim(S) \geq 2$. Then for some positive integer $k$, there exists a 
polynomial map $\phi \,:\, S \rightarrow \RR^k$ such that the closed convex hull of $\phi(S) \subset \RR^k$ is not the linear 
image of a spectrahedron.
\end{theorem}

These results show, among other examples, that there are high enough Veronese embeddings of semialgebraic sets that cannot be 
the projections of spectrahedra.

\begin{corollary} \cite[Corollary 4.24]{Scheiderer-HeltonNie}
Let $n,d$ be positive integers with $n \geq 3, d \geq 4$ or $n=2$ and $d \geq 6$. Let $m_1, \ldots, m_N$ be the non-constant 
monomials in $\RR[x]$ of degree at most $d$. Then for any semialgebraic set $S \subseteq \RR^n$ with non-empty interior, 
the closed convex hull of 
$$m(S) := \left\{ (m_1(s), \ldots, m_N(s)) \,:\, s \in S   \right\} \subset \RR^N$$
is not the linear image of a spectrahedron.
\end{corollary}

In contrast, Scheiderer had previously shown that all convex semialgebraic sets in $\RR^2$ have spectrahedral lifts \cite{Scheiderer-curves}, thus proving the Helton-Nie conjecture in the plane.
The current smallest counterexamples to the Helton-Nie conjecture are in $\RR^{11}$. Is it possible that there is a counterexample in 
$\RR^3$?

%%%%%%%%%%%%%%%%%%%%%%%%%%%%%%%%%%%%%%%%%%%%%%%%%%%%%%%%%%%%
\section{Notes}
There are many further results on spectrahedral lifts of convex sets beyond those mentioned here. 
An important topic that has been left out is that of symmetric spectrahedral lifts which are lifts that respect the symmetries of the convex set. Due to the symmetry requirement, such lifts are necessarily of size at least as large as the psd rank of the convex set. On the other hand, the symmetry restriction provides more tools to study such lifts and there are many beautiful results in this area, see \cite{FawziSaundersonParrilo-polygons}, \cite{FawziSaundersonParrilo-equivariantlifts}, 
\cite{FawziSaundersonParrilo-abeliangroups}.

Many specific examples of spectrahedral lifts of convex sets exist, and several of them have significance in applications. An easy general source  is the book \cite{SIAMBook}. In particular, Chapter 6 is dedicated to sdp representability of convex sets. This book includes a number of further topics in the area of {\em Convex Algebraic Geometry}. \\

{\bf Acknowledgments.} I am indebted to all my collaborators on the projects that contributed to this paper. 
I thank Pablo Parrilo for the construction in Example~\ref{ex:ball} and for several useful conversations. 
I also thank Hamza Fawzi and Jo\~ao Gouveia for comments on this paper, and Hamza for pointing out the  bound in Proposition~\ref
{prop:lower bound depending on n}. Claus Scheiderer and Daniel Plaumann were very helpful with the  
content of Section~\ref{sec:sos}.  
%%%%%%%%%%%%%%%%%%%%%%%%%%%%%%%%%%%%%%%%%%%%%%%%%%%%%%%%%%%%%%
\bibliographystyle{alpha}
\bibliography{icmpaper}
\end{document}